\documentclass[11pt]{amsart}

\usepackage{amsmath}
\usepackage{amssymb}
\usepackage{graphicx}

\newtheorem{theorem}{Theorem}[section]

\newtheorem{lemma}[theorem]{Lemma}

\theoremstyle{definition}

\numberwithin{equation}{section}
\usepackage{amsfonts,amsmath,amssymb,amsthm}
\usepackage{mathtools} 
\usepackage{comment}

\usepackage{enumitem} 

\usepackage{pgfplots}
\usepackage{tikz}
\usetikzlibrary{arrows}

\usepackage{pdfcomment}
\usepackage{todonotes}

\usepackage[mathscr]{euscript}

\newcommand{\M}{\mathcal{M}}

\newcommand{\N}{\mathbb{N}}
\newcommand{\Z}{\mathbb{Z}}

\newcommand{\R}{\mathbb{R}}

\numberwithin{theorem}{section} 

\theoremstyle{definition}
\numberwithin{equation}{section}

\usepackage{scalerel,stackengine}
\stackMath
\newcommand\reallywidehat[1]{%
\savestack{\tmpbox}{\stretchto{%
  \scaleto{%
    \scalerel*[\widthof{\ensuremath{#1}}]{\kern.1pt\mathchar"0362\kern.1pt}%
    {\rule{0ex}{\textheight}}
  }{\textheight}%
}{2.4ex}}%
\stackon[-6.9pt]{#1}{\tmpbox}%
}

\begin{document}
\title[Hardy-Littlewood Maximal Inequality]{A Direct Proof for Operator Hardy-Littlewood Maximal Inequality}

\author[C. Chuah]{Chian Yeong  Chuah}
\address[C. Chuah]{Department of Mathematics \\
The Ohio State University\\
231 West 18th Avenue\\
Columbus, OH 43210-1174, USA\\
 ORCid:0000-0003-3776-6555}
\email{{chuah.21@osu.edu}}

\author[Z. Liu]{Zhen-Chuan Liu} 
\address[Z. Liu]{Departamento de Matem\'aticas \\
Universidad Aut\'onoma de Madrid\\
C/ Francisco Tomás y Valiente, 7 Facultad de Ciencias, módulo 17, 28049 Madrid,Spain.\\
 ORCid: 0000-0002-6092-5473}
\email{{liu.zhenchuan@uam.es}}
	
  \author[T. Mei]{ Tao Mei}
  \address[T. Mei]{Department of Mathematics\\
Baylor University\\
1301 S University Parks Dr, Waco, TX 76798, USA.\\
 ORCid: 0000-0001-6191-6184}
\email{{tao\_mei@baylor.edu}}



\date{}
\maketitle

\begin{abstract}
We give a direct proof of the operator valued Hardy-Littlewood maximal inequality for $2\leq p<\infty$, which was first proved in \cite{Me07}. 
\end{abstract}

\section{Introduction}

The operator valued Hardy-Littlewood maximal inequality (\cite{Me07}) has become a basic tool in the  study of noncommutative analysis. See e.g.   \cite{CzXqYz13,HgLbWs21,PjReSm22} for the recent works which apply this inequality. 

The original proof of the inequality contained in \cite{Me07} reduces the problem to the martingale case where the noncommutative Doob's maximal inequality due to M. Junge (\cite{Ju02}) is applied. G. Hong  reproved this maximal inequality  in \cite{Hg13}. He follows Stein's idea of dominating the Hardy-Littlewood maximal function by maximal averages of heat semigroup operators and applies the noncommutative maximal ergodic theory due to M. Junge and Q. Xu (\cite{JuXq07}).  Both proofs are indirect which prevent  researchers, who are not familiar with the terminology of noncommutative analysis  from  a good understanding of the theorem. The purpose of this article is to provide a direct and more understandable  proof of this inequality for the case $p\geq 2$.

Recall that, for a locally integrable function $f$, the Hardy-Littlewood maximal function is defined as 
$$Mf(x)=\sup_t\frac1{2t}\int_{I(x,t)}|f(y)|dy,$$
where $I(x,t)$ is the interval centered at $x$ with length $2t.$ The classical Hardy-Littlewood maximal inequality states that 
\begin{eqnarray}
\|Mf\|_{L^p}\leq c\frac {p}{p-1}\|f\|_{L^p},
\end{eqnarray}
for all $f\in L^p(\R), 1<p\leq \infty$. 

Let $X$ be a Banach space, for X-valued functions $f$, their maximal function $Mf$ can be defined by considering the  maximal function of  the norm of $f$, 
\begin{eqnarray}
Mf(x)=\sup_t\frac1{2t}\int_{I(x,t)}\|f(y)\|_{X}dy.\label{Mf}
\end{eqnarray}
Apply the   classical Hardy-Littlewood maximal inequality to the $L_p$ function $\|f\|_X$, one obtains that 
\begin{eqnarray}
\|Mf\|_{L^p(X)}\leq c\frac {p}{p-1}\|f\|_{L^p(X)}=(\int_{\R} \|f\|_X^p)^{\frac1p},\label{HM2}
\end{eqnarray}
for all $f\in L^p(\R,X), 1<p\leq \infty$. A shortcoming is that  this type of   maximal function is scalar-valued  and may lose a lot of information of $X$ that $f$ originally carried. 

When $X=L^p(\Omega)$ with $\Omega$ a measurable space, one may define a $X$-valued  maximal function $F$ as  
\begin{eqnarray}
F(x,\omega)=\frac1{2t}\int_{I(x,t)}|f(y,\omega)|dy,\label{XM1}
\end{eqnarray}
and deduce from \eqref{HM2} that
\begin{eqnarray}
\|F\|_{L^p(\R,L^p( \Omega))}\leq c\frac {p}{p-1}\|f\|_{L^p(\R, L^p(\Omega))}.\label{HM3}
\end{eqnarray}

Note that in \eqref{XM1}, $F$ is a $X$-valued function that dominates the average of $|f|$ pointwisely, while in \eqref{Mf}, $Mf$ is merely a scalar valued function that carries much less information. Can  similar results like \eqref{XM1} and \eqref{HM3} hold for  $X$-valued functions when the Banach space $X$ is not equipped with a total  order but still has a reasonable partial order, e.g. for $X=$ the Schatten $p$ classes?
Based on  G. Pisier's work on operator spaces and M. Junge and Junge/Xu's work on the theory of  noncommutative martingales,   T. Mei (\cite{Me07}) proved a maximal inequality like \eqref{HM3}  for $X$ being the noncommutative $L^p$-spaces. In the case that  $X=$ the Schatten $p$-class $S_p$ and $f \in L_p( {S_p})$  for some $1< p\leq \infty$, this operator-maximal inequality says that there exists a   $S_p$-valued function $F$  such that 

(i) $\frac1{2t}\int_{x-t}^{x+t}|f(y)|dy\leq F(x)$ as operators for all $t>0$, i.e. $F-\frac1{2t}\int_{x-t}^{x+t}|f(y)|dy \in S^p$ is a self adjoint positive definite operator almost everywhere.

(ii) There exists an absolute constant $c $ such that  \begin{eqnarray}
\|F\|_{L^p(S^p)}\leq c\frac {p^2}{(p-1)^2}\|f\|_{L^p(S^p)}.\label{HM4}
\end{eqnarray}

A main obstacle for the proof of \eqref{HM4} is the lack of a total order. This was overcome by M. Junge using G. Pisier's  duality result for the operator spaces $L_p(\ell_\infty)$ and $L_p(\ell_1 )$. This article   aims   to 
give a direct proof for \eqref{HM4} and a more understandable proof of Pisier's duality  result (Lemma  \ref{lem:duality}) for  analysts who are not familiar with operator spaces.  
\section{Preliminary}


Let $H$ be a separable Hilbert space. We denote the space of bounded linear operators on $H$ by $B(H)$. For $x\in B(H)$, we denote by $x^*$ the adjoint operator of $x$, and define  $ |x|^p=(x^*x)^\frac p2$ by the functional calculus for $0<p<\infty$. We say that $x\in B(H)$ is self-adjoint if $x=x^*$. We say a self-adjoint operator $x$ is positive, denoted by $x\geq0$ if 
\begin{eqnarray}\label{geq0}
\langle xe,e\rangle\geq 0,
\end{eqnarray}
for all $e\in H$. This is equivalent to saying that $x=y^*y$ for some $y\in B(H)$. For two self-adjoint $x,y\in B(H)$, we write $x\leq y$ if $y-x\geq0$. 
The Schatten $p$-classes $S_p$, $0< p<\infty$ are the spaces of $x\in B(H)$ so that  
$$\|x\|_p=(tr |x|^p)^\frac1p<\infty.$$
Here, $tr$ is the usual trace $tr(x)=\sum_k\langle xe_k, e_k\rangle$ for $x\geq 0 \in B(H)$. The
Schatten $p$ classes share many properties with the $\ell_p$ spaces of sequences.  In particular, $S_p$ are Banach spaces for   $p\geq 1$ and $B(H)$ (resp. $S_q$)  is the dual space of $S_1$ (resp. $S_p$ for $1<p,q<\infty, \frac1p+\frac1q$),   via the isometric isomorphism,
$$x\mapsto \phi_x: \phi_x(y)=tr(xy)$$
for $y\in S_1$ (resp. $S_p$). 

\subsection{Noncommutative $L_p$ spaces} A {\it von Neumann algebra}, by definition, is  a weak-$*$ closed subalgebras $\M$ of $B(H)$. 
The completeness according to the weak $*$-topology of $\M$ ensures that it contains  the spectral projections of its self-adjoint elements.  $\ell_\infty({\Bbb N})$ which is isometrically isomorphic to the subalgebras of the diagonal operators and $B(H)$ itself are two basic examples of von Neumann algebras.  
The usual trace $\tau=tr$ on $B(H)$ is a linear functional on the weak-$*$ dense subspace $S_1$ satisfying the following properties,
\begin{itemize}
\item[i)] Tracial: $\tau(xy)=\tau(xy)$,

\item[ii)] Faithful: if $x \ge 0$ and $\tau (x)=0$ then $x=0$,

\item[iii)] Lower semi-continuous: $\tau (\sup x_i)=\sup\tau (x_i)$ when $x_i\geq0$ is increasing,

\item[iv)] Semifinite: for any $x\geq 0$, there exists $0 \le y \leq x$ such that $\tau (y) < \infty$.
\end{itemize}
This leads to defining {\it semifinite von Neumann algebras} $\M$ as those equipped with a trace $\tau$, which is an unbounded linear   functional satisfying (i)-(iv) for $x,y$ belonging to a weak $*$ dense subspace.  Note the restriction of the usual trace of $B(H)$ on $\M$ may not be semifinite, and not every von Neumann algebras is semifinite. Given such a pair $(\M,\tau)$ (which  is usually viewed   as a noncommutative $L_\infty$- space) the {\it noncommutative $L_p$ spaces} associated to it are the completion of $f\in{\mathcal M}$ with finite quasi norm $$\|f\|_p \, = \, \left[ \tau \left( |f|^p \right) \right]^\frac1p \quad \mbox{for} \quad 0< p < \infty,$$ where $|f|^p = (f^*f)^{p/2}$ is constructed via functional calculus. We set $L_\infty(\M)=\M$.

The Schatten $p$-classes $S_p$ and the $L_p$ spaces on a semifinite measure space $(\Omega,\mu)$  are examples of noncommutative $L_p$ spaces associated with $\M = B(H)$ and ${\M}=L_\infty(\Omega,\mu)$ respectively. Another basic example arises from group von Neumann algebras. Every commutative semifinite von Neumann algebra is isometrically isomorphic to the space $L_\infty(\Omega,\mu)$ of essentially bounded functions on some semifinite measure space $(\Omega,\mu)$. 
Many basic  properties of $L_p(\Omega,\mu)$ extend to $L_p(\M)$. In particular, one has the H\"older's inequality which states that
\begin{eqnarray*}
\|xy\|_p\leq \|x\|_r\|y\|_q,
\end{eqnarray*}
for $x\in L_r(\M),y\in L_q(\M), 0\leq p,q,r,\leq \infty, \frac1q+\frac1r=\frac1p$. The  interpolation properties of $L_p(\Omega,\mu)$ extend to $L_p(\M)$ as well, and the duality  properties of $L_p(\Omega,\mu)$ extend to $L_p(\M)$ for the range $p\geq 1$. The elements in $L_p(\M)$ may not belong to $\M$  or $B(H)$ in general. They can be understood as unbounded operators affiliated with $\M$. 
For self-adjoint elements $x,y$ in $L_p(\M)$, we  say $x$ is positive, denoted by  $x\geq 0$   if \eqref{geq0} holds (the quantity may be $\infty$ though). We write $x\leq y$ if $y-x\geq0$.

Define $\ell_\infty(L_{p}(\mathcal{M}))$ as the space of all sequences $x=(x_{n})_{n\geq 1}$ in $L_{p}(\mathcal{M})$ such that
\begin{eqnarray}
\|x\|_{\ell_{\infty}(L_{p}(\mathcal{M}))}= \sup_n\|x_n\|_p<\infty. \label{linftyp}
\end{eqnarray}
  Define $\ell_1(L_{p}(\mathcal{M}))$ as the space of all sequences $x=(x_{n})_{n\geq 1}$ in $L_{p}(\mathcal{M})$ such that
\begin{eqnarray}
\|x\|_{\ell_{1}(L_{p}(\mathcal{M}))}=  \sum_n\|x_n\|_p<\infty\label{l1p}.
\end{eqnarray}
We have the duality that 
\begin{eqnarray*}
 \ell_{\infty}(L_{q}(\mathcal{M}))=(
\ell_{1}(L_{p}(\mathcal{M})))^*
\end{eqnarray*}
for $1\leq p<\infty, \frac1p+\frac1q=1$, via the isometric isomorphism 
$$x\mapsto \varphi_x; \varphi_x(y)=\sum_n \tau (x_ny_n).$$
On the other hand,  $
\ell_{1}(L_{q}(\mathcal{M}))$ isometrically embeds into the dual of $\ell_{\infty}(L_{p}(\mathcal{M}))$ as  a weak $*$-dense subspace via the same isomorphism for the same range and relation of $p,q$. We refer to the survey paper \cite{PiXq03} for more information on noncommutative $L_p$ spaces.

  \subsection{Pisier's $L_p(\M;\ell_{\infty})$ norm}
  
  Given two positive operators $x,y\in L_p(\M)$ e.g. $ x,y\in S_p$, the expression $\sup(x,y)$ does not make sense unless $x,y$ commutes so that the least upper element exists. 
  Nevertheless, the following is a reasonable expression for $\|\sup_n x_n\|_{L_p(\M)}$ for   sequences of positive elements $x_n$ in $L_p(\M)$,
  \begin{eqnarray}
  \||(x_n)\|_{L_p(\M;\ell_\infty^+)}=\inf\{\|a\|_{L_p}; x_n\leq a, a\in L_p(\M)\}.\label{supxn}
  \end{eqnarray}
 For   sequences of positive elements $y_n$ in $L_p(\M), 1\leq p<\infty$ such that $\sum_{n=1}^N y_n$  converges in $L_p(\M)$, let
  \begin{eqnarray}
  \|(y_n)\|_{L_p(\M;\ell_1^+)}= \left\| \sum_{n=1}^\infty y_n \right\|_{L_p(\M)}.\label{ell1xn}
  \end{eqnarray}
  Here the right hand side is an  increasing sequence. 
  G. Pisier (\cite{Pi98} and M. Junge (\cite{Ju02} proved that the expressions \eqref{supxn} and \eqref{ell1xn}  extend to   Banach space norms.\footnote{The obvious extension $\|\sum_n|x_n|\|_p$ for sequences $x_n\in L_p(\M)$ fails the triangle inequality.} Furthermore, they proved a duality result for these two  norms (see Lemma \ref{lem:duality}), which is the key for Junge's proof of the noncommutative Doob's maximal inequality.

Define $L_{p}(\mathcal{M};\ell_{\infty})$ as the space of all sequences $x=(x_{n})_{n\geq 1}$ in $L_{p}(\mathcal{M})$ which admits a factorization 
\begin{equation}\label{eq:factorization}
x_{n}=az_{n}b, \quad \forall n\geq 1,
\end{equation}
where $a,b\in L_{2p}(\mathcal{M})$ and $z=(z_{n})$ belongs to the unit ball of  $\M$.
Given $x\in L_{p}(\mathcal{M};\ell_{\infty})$, define
\begin{eqnarray}
\|x\|_{L_{p}(\mathcal{M};\ell_{\infty})}= \inf \left\{ \|a\|_{2p} \sup_{n\geq 1}\|z_n\|_{\infty} \|b\|_{2p}  \right\},\label{Lpinfty}
\end{eqnarray}
where the infimum is taken over all factorizations of $x$ as in \eqref{eq:factorization}. We denote the subset of $L_{p}(\mathcal{M};\ell_{\infty})$ consisting of all positive sequences by $L_{p}(\mathcal{M};\ell^+_{\infty})$. Next, we denote the subspace of $L_{p}(\mathcal{M};\ell_{\infty})$ consisting of all  sequences $x=(x_{n})_{n}$ such that $x_n=0$ for $n>N$ by $L_{p}(\mathcal{M};\ell^N_{\infty})$.

Define $L_{p}(\ell_{1})$ as the space of all sequences $x=(x_{n})_{n\geq 1}$ in $L_{p}(\mathcal{M})$ which admits a factorization
\begin{equation} \label{anbn}
x_{n}=a_{n}b_n, \quad \text{ for all } n\geq 1,
\end{equation}
such that the series $\sum_n a_na_n^*$ and $\sum_n b_n^*b_n$ converges  in $L_{p}(\mathcal{M}).$ Given $x\in L_{p}(\mathcal{M};\ell_{1})$, define
\begin{equation}  \label{Lpell1}
\|x\|_{L_{p}(\mathcal{M};\ell_{1})}   = \inf\left\{ \left\|\left( \sum a_{j}a_{j}^{*} \right)^{\frac{1}{2}}\right\|_{2p} \cdot \left\| \left(\sum b_{j}^{*}b_{j}\right)^{\frac{1}{2}}\right\|_{2p} \right\},
\end{equation}
where the infimum is taken over all possible factorizations \eqref{anbn}.\footnote{There are slightly different definition of the space $L_{p}(\mathcal{M};\ell_{1})$ in the literature. Here, we use the definition given in \cite{Pi10}.}  We denote the subset of $L_{p}(\mathcal{M};\ell_1)$ consisting of all positive sequences by $L_{p}(\mathcal{M};\ell^+_1)$. 
Define $L_{p}(\ell_{1}^N)$ to be the space of   all sequences $x=(x_{n})_{n\geq 1}\in L_{p}(\ell_{1})$ with $x_n=0$ for $n>N$.
G. Pisier (\cite{Pi98}) proved that  \eqref{Lpinfty} and \eqref{Lpell1} are norms extending \eqref{supxn} and \eqref{ell1xn}.


\begin{lemma}[Pisier \cite{Pi98}]\label{lem:sum}

For sequences $x_n\geq 0\in L_p(\M)$,  we have 
\begin{eqnarray}
\|(x_n )\|_{L_{p}(\mathcal{M};\ell_{1})}=\|(x_n )\|_{L_{p}(\mathcal{M};\ell_{1}^+)}\label{eq1}
\end{eqnarray}
in the sense that both sides are equally finite or both sides are infinite, and 
\begin{eqnarray}
\|(x_n )\|_{L_{p}(\mathcal{M};\ell_{\infty} )}\leq\|(x_n )\|_{L_{p}(\mathcal{M};\ell_{\infty}^+ )}\leq 4\|(x_n )\|_{L_{p}(\mathcal{M};\ell_{\infty} )}\label{eqinfty}
\end{eqnarray}
\end{lemma}

\begin{proof} 

We prove \eqref{eq1} first. The left hand side is obviously smaller because we may choose 
$a_n=b_n=(x_n)^\frac12$ for $x_n\geq0$. To prove that the right hand side is smaller, we assume $\|(x_n )\|_{L_{p}(\mathcal{M};\ell_{1})}=1$ and assume that there exists a factorization that $x_n=a_nb_n$ such that  
$$\|\sum_{n} a_na_n^*\|_p , \|\sum_nb_nb_n^*\|_{p}\leq 1+\varepsilon .$$ 
Then,   $2x_n=a_nb_n+b_n^*a_n^*\leq a_na_n^*+b_n^*b_n$ because 
$(a_n^*-b_n)^*(a_n^*-b_n)\geq0$. So, $\sum_{n=1}^{N}x_{n}$ converges in $L_p({\M})$ because $\sum_{n}^N a_na_n^*$ and $\sum_{n}^N b_n^*b_n$ do. Moreover, 
$$2\|\sum_{n=1}^{N}x_{n}\|_{p}\leq \|\sum_{n=1}^{N}a_na_n^*+b_nb_n^*\|_{p}\leq 2+2\varepsilon.$$
We conclude   by taking $N\rightarrow\infty,\varepsilon\rightarrow 0.$ 

For \eqref{eqinfty}, assuming $x_n\leq a\in L_p(\M)$, we denote by $p$ the projection onto the kernel of $a$. Then, $p\in \M$. Let $z_n=(p+ a^\frac12)^{-1}x_n(p+ a^\frac12)^{-1}$. Then, $z_n$ belongs to the unit ball of $\M$ and $x_n=a^\frac12z_na^\frac12$. 
We see that the first inequality holds. Next, we show the second inequality. We assume that $\|(x_n )\|_{L_{p}(\mathcal{M};\ell_{\infty} )}=1$ and $x_n$ has a factorization $x_n=az_nb$ with $\|z_n\|=1$ and $\|a\|_{2p},\|b\|_{2p}\leq 1+\varepsilon$. 
We write $z_n= \sum_{k=0}^3 i^kz_{n,k} $ with contractions $z_{n,k}\geq 0$, and consider the new decomposition $x_n=\sum_k a_kz_{n,k} b $ with $a_k=i^ka$.  Noting that $(a_k^*-b)^*z_{n,k}(a_k^*-b)\geq0$, we have
$$x_n=\frac12\sum_{k=0}^3(a_kz_{n,k}b+b^*z_{n,k}a_k^*)\leq \frac12\sum_{k=1}^3(a_k^*z_{n_k}a_k+bz_{n,k}b^*)\leq 2(a^*a+bb^*),$$
with $ \|2(a^*a+bb^*)\|_p\leq (2+2\varepsilon)^2.$
Taking $\varepsilon\rightarrow0$, we conclude \eqref{eqinfty}.
\end{proof}

The following lemma is another key to understanding the proof of the operator Hardy-Littlewood maximal inequality. 
The result was proved by G. Pisier (\cite{Pi98, Ju02}). We include an argument for the case of finite sequences below. 

\begin{lemma}[\cite{Pi98, Ju02}]\label{lem:duality}
The norms \eqref{supxn} and \eqref{ell1xn} are in duality. More precisely, for $1\leq p<\infty, \frac1p+\frac1q=1$,
\begin{itemize}
\item[(i)] For any $N$-tuple $(y_{1},\dots,y_{N})$ in $L_{p}(\mathcal{M})$ and $y_k\geq 0,  k=1,\dots,N$, we have 

\begin{equation}\label{eq:dualityNorm}
\|(y_{n})\|_{L_{p}(\mathcal{M};\ell_{1}^N)}= \sup\left\{ |\tau\left( \sum_{j=1}^N x_{j}y_{j} \right) |:  \|(x_{j})\|_{L_{q}(\mathcal{M};\ell_{\infty}^+)}\leq 1\right\}.
\end{equation}
\item[(ii)] For any  bounded sequence $(x_n)$ in $L_{q}(\mathcal{M})$ with $x_n\geq 0$, we have  
\begin{equation}\label{eq:dualityNorm2}
\|(x_{n})\|_{L_{q}(\mathcal{M};\ell_{\infty})}= \sup_N\left\{ |\tau\left( \sum_{j=1}^N x_{j}y_{j} \right) |: y_j\geq 0, \|(y_{j})\|_{L_{p}(\mathcal{M};\ell_{1}^{N})}\leq 1\right\}.
\end{equation}

\item [(iii)] $L_{q}(\mathcal{M};\ell_{\infty}^+)$  embeds isometrically into the dual space of $L_{p}(\mathcal{M};\ell_{1})$ for $1\leq p<\infty$ via the isomorphism
$$x\mapsto \varphi_x: \varphi_x(y)=\sum_n x_ny_n.$$
For any $\varphi$ in $L_{p}(\mathcal{M};\ell_{1})^*$ such that $\varphi((y_n))\geq 0$ for  finite positive sequences $(y_n)\in L_p(\M)$, there is a (unique) positive sequence $(x_{n})$ in $L_{q}(\mathcal{M})$ with
$$
\|x\|_{L_{q}(\mathcal{M};\ell_{\infty})}= \|\varphi\|_{(L_{p}(\mathcal{M};\ell_{1}))^{*}}
$$
such that for any $N\geq 1$ and any $y=(y_{1},\dots,y_{N})\in L_{p}(\mathcal{M};\ell_{1}^N)$,
$$
\varphi(y)= \tau\left( \sum_{j=1}^{N} x_{j}y_{j} \right).
$$
\end{itemize}
\end{lemma}
\begin{proof}
(i).	Let  $x=(x_{n})\in L_{q}(\mathcal{M};\ell_{\infty}), y=(y_{n})\in L_{p}(\mathcal{M};\ell_1)$. First, we prove that 
\begin{equation}\label{eq:duality}
\left|\sum \tau(x_{j}y_{j})\right|\leq \|(x_{j})\|_{L_{q}(\mathcal{M};\ell_{\infty})}\|y\|_{L_{p}(\mathcal{M};\ell_{1}^N)}.
\end{equation}
Consider a factorization $x_{j}=a z_{j}b$ where $a,b\in L_{2q}(\mathcal{M})$ and $(z_{j})$ belongs to the unit ball of $\M$.  
Also consider  a factorization of $y_{j}=u_{j}v_{j}$ where $u_{j},v_{j}\in L_{2p}(\mathcal{M})$.
Then, by H\"older's inequality and the Cauchy-Schwarz inequality,
\begin{align*}
\left| \sum_{j}\tau(az_jb u_{j}v_{j})\right| &=\left| \sum_{j}\tau(z_jb u_{j}v_{j}a)\right| \leq \sum_{j}\|bu_{j}v_{j}a\|_{1} \\
&\leq \sum_{j}\|bu_{j}\|_{2} \|v_{j}a\|_{2} \\
&\leq \left(\sum_{j}\|bu_{j}\|_{2}^2\right)^{\frac{1}{2}}\left( \sum_{j}\|v_{j}a\|_{2}^2 \right)^{\frac{1}{2}} \\
&=\left(\tau (b  \sum_{j}u_{j}u_{j}^*b^*)\right)^{\frac12}\left(\tau (a^*  \sum_{j}v_{j}v_{j}^*a)\right)^{\frac12} \\
&\leq \|b^*b\|^\frac12_{q}\|\sum_{j}u_{j}u_{j}^*  \|^\frac12_{p}\|aa^*\|^\frac12_{q}\|\sum_{j}v_{j}v_{j}^*  \|^\frac12_{p}\\
&\leq \|b\| _{2q}\|(\sum_{j}u_{j}u_{j}^*  )^\frac12\|_{p}\|a \| _{2q}\|(\sum_{j}v_{j}v_{j}^* )^\frac12\|_{p}
\end{align*}
Hence, we proved the one side inequality of (i) and (ii) and the first half of (iii). 

Now, suppose   $y_n\geq 0, y=(y_n)_n \in L_{p}(\mathcal{M};\ell_{1}^N)$.
Choose $x=(x_{n})_{n=1}^N$ with $x_{n}=(\sum_{k=1}^{N}y_{k})^{p-1}$ for all $1\leq n\leq N$. Note 
$$
\|(x_{n})\|_{L_{q}(\mathcal{M};\ell_{\infty})}=\left\|\left( \sum_{k}y_{k} \right)^{p-1}\right\|_{q}=\|\sum_{k=1}^N~y_{k}\|_{p}^{p-1}
$$
Thus,
$$
 \tau\left( \sum_{k}y_{k}x_{k} \right)=\tau\left( \sum_{k=1}^N~y_{k}\left( \sum_{k=1}^N~y_{k} \right)^{p-1} \right)=\|\sum_{k=1}^N~y_{k}\|_{p}.
$$
Therefore, we proved the other direction of (i).

We now prove the other direction of (ii). By the Hahn-Banach theorem, for any $(x_{n})\in L_{p}(\mathcal{M};\ell_{\infty})$, there exists $\varphi\in L_{p}(\mathcal{M};\ell_{\infty})^{*}$, such that $\|\varphi\|=1$ and $\varphi((x_{n}))=\|(x_n)\|_{L_{p}(\mathcal{M};\ell_{\infty})}$.
Since $L_{p}(\mathcal{M};\ell_{\infty})$ is a subspace of $\ell_{\infty}(L_{p}(\mathcal{M}))$,  there exists  $\tilde\varphi\in (\ell_{\infty}(L_{p}(\mathcal{M})))^*$ such that 
$\varphi(x)=\tilde\varphi(x)$. Since  the unit ball of $\ell_{1}^N(L_{p}(\mathcal{M}))$ is weak $*$-dense in the unit ball of $(\ell_{\infty}(L_{p}(\mathcal{M})))^*$ and $x_n\geq0$, we conclude that for any $\varepsilon>0$, there exists a $\varphi_\varepsilon\in (\ell_{\infty}(L_{p}(\mathcal{M})))^*$ in the form
\[
   \varphi_{\varepsilon}((x_{n}))=\sum_{n=1}\tau(x_{n}y_{n}),
\]
such that $\varphi(x)=\tilde\varphi(x)\leq \varphi_\varepsilon(x)+\varepsilon$ and   $(y_n)_n\in \ell_{1}^N(L_{p}(\mathcal{M}))$  with $y_n\geq 0$.  On the other hand,   we know  from (i) that,
\begin{eqnarray*}
\|(y_{n})\|_{L_{p}(\mathcal{M};\ell_{1})}&=&\sup_{\|x\|_{L_p(\M;\ell_\infty)}\leq1}|\tau(x_{n}y_{n})|\\
&\leq&\sup_{\|x\|_{ \ell_\infty(L_p)}\leq1}|\tau(x_{n}y_{n})|\\
&\leq &\|\varphi_\varepsilon \|= 1.
\end{eqnarray*}
We obtain
\begin{equation}
\|(x_{n})\|_{L_{p}(\mathcal{M};\ell_{\infty})}=\varphi(x)\leq \varphi_{\varepsilon}(x)+\varepsilon\leq \sup \left\{ |\tau\left( \sum x_{n}y_{n} \right) |: y_j\geq 0, \|(y_{n})\|_{L_{p^{'}}(\mathcal{M};\ell_{1}^{N})}\leq 1\right\}+\varepsilon.
\end{equation}
We then conclude (ii) by letting $\varepsilon\rightarrow0$.

We now prove the other direction of (iii). 
Note that $\ell_1(L_p(\M))$ is a sub-linear vector space of $L_{p}(\mathcal{M};\ell_{1})$ equipped with a larger norm. So, for any bounded linear functional $\varphi\in (L_{p}(\mathcal{M};\ell_{1}))^*$, its  restriction on $\ell_{1}(L_{p}(\mathcal{M}))$ defines a  bounded linear functional $\tilde\varphi\in (\ell_{1}(L_{p}(\mathcal{M})))^*=\ell_{\infty}(L_{q}(\mathcal{M}))$.  We conclude that  there exists $x_n\geq 0, (x_n)\in \ell_{\infty}(L_{q}(\mathcal{M}))$ such that
\begin{eqnarray}\label{phiyn}
   \varphi ((y_{n}))=\sum_{n=1}^\infty\tau(x_{n}y_{n}),
\end{eqnarray}
 for all $(y_n)\in \ell_1(L_p(\M))\subset L_{p}(\mathcal{M};\ell_{1})$. In particular, the expression \eqref{phiyn} holds  for any finite sequences $(y_n)\in L_{p}(\mathcal{M};\ell_{1})$.   By (ii), we have
 \begin{eqnarray*}
  \|(x_n)\|_{L_{q}(\mathcal{M} ;\ell_{\infty})}  &=&\sup  \left\{\tau \left(\sum_{n=1}x_{n}y_{n} \right); {\rm finite\ sequences} (y_n), \|y_n\| _{ L_{p}(\mathcal{M};\ell_{1})}\leq1\right\}\\
 & =&\sup\left\{ \varphi((y_n)); {\rm finite\ sequences} \ (y_n), \|y_n\| _{ L_{p}(\mathcal{M};\ell_{1})}\leq1\right\}\leq \|\varphi\|_{(L_{p}(\mathcal{M};\ell_{1}))^*}.
\end{eqnarray*}
\end{proof}

\section{Operator-Maximal Inequality}
Let $\M$ be a semifinite von Neumann algebra, e.g. $ L_\infty(\Omega)$ or $B(H)$. Let $L_p(\M)$ be the associated noncommutative $L_p$ spaces, e.g. $L_p(\Omega)$ or the Schatten classes $S_p$.
Let  $L_p(\R,\M), 1\leq p<\infty$ be the space of all $L_p(\M)$-valued Bochner-measurable functions $f$ on the real line such that 
$$\|f\|_{L^p(\R,L_p(\M))}= \left(\int_\R \|f(x)\|_p^p dx \right)^\frac1p= \left(\int_\R \tau \left[ |f(x)|^p \right] dx \right)^\frac1p<\infty.$$ 
We prove the following operator Hardy-Littlewood maximal inequality for $f\in L_p(\R,L_p(\M))$ where $2\leq p<\infty$. The corresponding result for $p=\infty$ is trivial. 
\begin{theorem}\label{thm:main}

	Given $f \in L_p( \R, L_p(\M))$  for some $2\leq  p<\infty$,   there exists a   $L_p(\M)$-valued Bochner-measurable function $F$  such that 

(i) $\frac1{2t}\int_{x-t}^{x+t}|f(y)|dy\leq F(x)$ as operators for all $t>0$, i.e. $F-\frac1{2t}\int_{x-t}^{x+t}|f(y)|dy\geq0  $ almost everywhere.

(ii) There exists an absolute constant $c $ such that  \begin{eqnarray}
\|F\|_{L^p(\R,L_p(\M))}\leq c\|f\|_{L^p(\R,L_p(\M))}.\label{HM5}
\end{eqnarray}

\end{theorem}

In order to prove main theorem, we prove the dual form of Theorem \ref{thm:main}. Let $\mathcal{N}$ be the von Neumann algebra tensor product $L_{\infty}(\mathbb{R})\otimes \mathcal{M}$ equipped with the semifinite trace  $\nu=\int\otimes \tau$. Then, $L_p(\R,\M)$ coincides with the noncommutative $L_p$ spaces $L_p({\N})$ associated with the pair $({\N},\nu)$ for $1\leq p<\infty$. Let $T_n, n>0$ be the averaging operator on $L_p(\R)\otimes L_p(\M)$ defined by
\begin{eqnarray*}
(T_{n}f)(x)=\frac{1}{2^{n+1}}\int_{x-2^n}^{x+2^n}f(t)  \, dt
\end{eqnarray*}
It is easy to verify that $ \{T_{n}\}_{n\in\mathbb{Z}}$ is a family of operators from $L_{2}(\mathcal{N})\to L_{2}(\mathcal{N})$ satisfying 
\begin{itemize}
  \item	$T_{n}=T_{n}^*$;
   \item $T_{n}g \geq 0$ if $g\geq 0$;
   \item   $T_{n}T_{m}\leq 2T_{\sigma(m)}$ for $\sigma(m)=m+1$ and any $n\leq m$.
\end{itemize}

\begin{lemma}\label{lemma3.1} For any finite sequence $g_{n}\in L_{2}(\mathcal{N})$ with all $g_n\geq 0$, we have
\begin{equation}\label{lem:Stein}
\|(T_{n}g_{n})\|_{L_{2}(\mathcal{N};\ell_1)}\leq 4\|(g_{n})\|_{L_{2}(\mathcal{N};\ell_1)}.
\end{equation}
\end{lemma}

\begin{proof}
Given a positive sequence $(g_n)_n\in L_2({\N})$ with only  finitely many non-zero terms and     a bijection $\alpha$ on $\Z$, we have that for $\sigma(m)=m+1$,
\begin{align*}
\left \|\sum_{n} T_{n}g_{\alpha(n)} \right\|_{L_{2}(\mathcal{N})}^{2}&= \nu\left( \sum_{n,m}T_{n}g_{\alpha(n)}T_{m}g_{\alpha(m)} \right)\\
&=\nu\left( \sum_{n< m}T_{n}g_{\alpha(n)}T_{m}g_{\alpha(m)} \right)+ \nu\left( \sum_{n\geq m}T_{n}g_{\alpha(n)} T_{m}g_{\alpha(m)}\right) \\
({\rm by\ the\ tracial\ property\ of}\  \nu)&=\nu\left( \sum_{n< m}T_{n}g_{\alpha(n)}T_{m}g_{\alpha(m)} \right)+ \nu\left( \sum_{n\geq m} T_{m}g_{\alpha(m)}T_{n}g_{\alpha(n)}\right) \\
&=\nu\left( \sum_{n< m}T_{n}g_{\alpha(n)}T_{m}g_{\alpha(m)} \right)+ \nu\left( \sum_{n\leq m} T_{n}g_{\alpha(n)}T_{m}g_{\alpha(m)}\right) 
\end{align*}
Note that  $\nu\left( \sum_{n}T_{n}g_{\alpha(n)}T_{n}g_{\alpha(n)} \right)\geq0$. So, we have that
\begin{align*} 
\left\| \sum_{n} T_{n}g_{\alpha(n)} \right\|_{L_{2}(\mathcal{N})}^{2}
&\leq 2\nu\left( \sum_{n\leq m}T_{n}g_{\alpha(n)}T_{m}g_{\alpha(m)} \right)\\&=2\nu\left( \sum_{n\leq m} g_{\alpha(n)}T_{n}T_{m}g_{\alpha(m)} \right) \\
&\leq 4~\nu\left( \sum_{n\leq m}g_{\alpha(n)}T_{\sigma(m)}g_{\alpha(m)} \right).
\end{align*} 
By the tracial property of $\nu$, we have that $\nu (ab)=\nu( b^\frac12ab^\frac12)\geq0$ for any $a,b\geq0$. So
\begin{align*} 
\left\| \sum_{n} T_{n}g_{\alpha(n)} \right\|_{L_{2}(\mathcal{N})}^{2}&\leq 4~\nu\left( \sum_{n,m}g_{\alpha(n)}T_{\sigma(m)}g_{\alpha(m)} \right) \\  &\leq4~\nu\left( \left( \sum_{n}g_{\alpha(n)} \right)\left( \sum_{m}T_{\sigma(m)}g_{\alpha(m)} \right) \right)\\
&=4~\nu\left( \left( \sum_{n}g_n) \right)\left( \sum_{m}T_mg_{\sigma^{-1}\alpha(m)} \right) \right)\\
&\leq 4~ \left\|\sum_{n}g_{n}\right\|_{L_{2}(\mathcal{N})}~\left\|\sum_{m}T_{m}g_{\sigma^{-1}\alpha(m)}\right\|_{L_{2}(\mathcal{N})}.
\end{align*}
Now, taking the supremum over all bijections $\alpha$ on both sides, we get 

\begin{align*} 
\sup_\alpha \left\| \sum_{n} T_{n}g_{\alpha(n)} \right\|_{L_{2}(\mathcal{N})}^{2} &\leq 4 \left\| \sum_{n}g_{n} \right\|_{L_{2}(\mathcal{N})}\sup_\alpha \left\| \sum_{m}T_{m}g_{\sigma^{-1}\alpha(m)} \right\|_{L_{2}(\mathcal{N})}\\
&= 4 \left\| \sum_{n}g_{n} \right\|_{L_{2}(\mathcal{N})}\sup_\alpha \left\| \sum_{m}T_{m}g_{\alpha(m)} \right\|_{L_{2}(\mathcal{N})}
\end{align*}
By dividing the finite number $\sup_\alpha\|\sum_{m}T_{m}g_{\alpha(m)}\|_{L_{2}(\mathcal{N})}$ on both sides, we get \eqref{lem:Stein}.
\end{proof}

Note that $ \{T_{n}\}_{n\in\mathbb{Z}}$ is a family of positive-preserving contractions on $L_{1}(\mathcal{N})$. Lemma \ref{lemma3.1} holds trivially if we replace $L_2({\N})$ with $L_1({\N})$. We  show that this remains true if we replace $L_2({\N})$ with $L_p({\N})$ for all $1<p<2$.
We need the following Cauchy-Schwartz inequality. We include a proof for completeness.

\begin{lemma}\label{lem:holder}
 Suppose $a_n\in L_q(\mathcal{N},\nu), b_n\in L_r(\mathcal{N},\nu)$. Then,
  we have
  \begin{equation}
    \left\|\sum_{n=1}^{N} T_n(a_n^* b_n)\right\|_p \leq \left\|\sum_{n=1}^{N}T_n(a_n^* a_n)\right\|_q~\left\|\sum_{n=1}^{N}T_n(b_n^* b_n)\right\|_r.
  \end{equation}
  In particular, 
  \begin{equation}
    \left\|\sum_{n=1}^{N} T_n(a_n^* b_n)\right\|_p \leq \left\|\sum_{n=1}^{N}T_n(a_n^* a_n)\right\|_p~\left\|\sum_{n=1}^{N}T_n(b_n^* b_n)\right\|_p.
  \end{equation}
\end{lemma}
\begin{proof}
  Let $X_n=\begin{pmatrix}
    a_n & b_n\\
    0 & 0
  \end{pmatrix}$. Then we have 
  \[
  X_n^* X_n= \begin{pmatrix}
    a_n^* a_n & a_n^* b_n\\
    b_n^* a_n & b_n^* b_n
  \end{pmatrix}.
  \]This implies that \[
   \begin{pmatrix}
    \sum_{n=1}^N T_n(a_n^* a_n)   & \sum_{n=1}^{N} T_n(a_n^* b_n) \\
    \sum_{n=1}^N T_n(b^* a_n)    & \sum_{n=1}^N T_n(b_n^* b_n)
  \end{pmatrix}=  \begin{pmatrix}
      \alpha & \gamma\\
      \gamma^* & \beta
    \end{pmatrix}\geq 0.
  \]
 Then, by \cite[Prop. 1.3.2]{Br07}, there exists a contraction $y$ such that $\gamma=\alpha^{\frac{1}{2}} y \beta^{\frac{1}{2}}$. 
  Thus by H\"older's inequality,
  \[
  \left\| \sum_{n=1}^N T_n(a_n^* b_n) \right\|_p = \|\gamma\|_p= \|\alpha^{\frac{1}{2}} y \beta^{\frac{1}{2}}\|_p\leq \|\alpha^{\frac12}\|_q \|\beta^{\frac12}\|_r.
  \]
  The lemma is proved.
\end{proof}

\begin{lemma}\label{lemma:p12}Under the same assumption of Lemma \ref{lemma3.1}, we have that,
\begin{equation}
  \|(T_n g_n)\|_{L_p(\mathcal{N};\ell_1)}\leq 4^{2-\frac{2}{p}} \|(g_n)\|_{L_p(\mathcal{N};\ell_1)}.
\end{equation}
for all finite sequences $(g_n)\in L_p(\mathcal{N};\ell_1)$, $g_n\geq 0$, $1\leq p\leq 2$.
\end{lemma}
\begin{proof}
  Assume that $\|\sum_{n=1}^N~g_{n}\|_{p}=1$.  Let $g=\left( \sum_{n=1}^N g_{n}\right)^{1/2}$.  We then have $\|g\|_{2p}= 1.$ By approximation, we may assume $g$ is invertible.
Let $h_{n}=g^{-1}g_{n}g^{-1}\geq 0$. Then, $\|\sum_{n=1}^{N}h_{n}\|_{\infty}\leq 1$. Let $\theta$ be defined by $\frac{1}{p}=\frac{1-\theta}{1}+\frac{\theta}{2}$. Let 
$$
F(z)=g^{(1-z)p+zp/2}h_{n}g^{(1-z)p+zp/2}
$$
and $U_{t}=g^{-ipt/2}$.  Note that $F(\theta)=g_{n}$ and $F(it)=U_{t}g^p h_{n} g^{p}U_{t}$, $F(1+it)=U_{t}g^{p/2}h_{n}g^{p/2}U_{t}$. 
Therefore,

\begin{align}
\left\| \sum_{n=1}^{N}T_{n}F(it) \right\|_{1}& \leq \sum_{n=1}^N \left\|T_{n}F(it) \right\|_{1} \leq \sum_{n=1}^N \left\| F(it) \right\|_{1}\leq \sum_{n=1}^N\|g^ph_{n}g^p\|_{1}\nonumber\\ 
&=\nu\left(g^p\sum_{n=1}^Nh_{n}~g^p\right) \nonumber\\
&\leq \nu(g^{2p})=1.
\end{align}
On the other hand,   applying Lemma \ref{lem:holder} to $a_n^*= U_t g^{\frac{p}{2}}h_n^{\frac{1}{2}}, b_n=h_n^{\frac12} g^{\frac{p}{2}}U_t$
\begin{align}
& \left\| \sum_{n=1}^{N} T_{n}F(1+it) \right\|_{2}= \left\| \sum_{n=1}^N~T_{n}\left( U_{t}g^{\frac{p}{2}}h_{n}^{\frac{1}{2}} h_n^{\frac12} a^{\frac{p}{2}}U_{t} \right) \right\|_{2} \nonumber \\
&\leq \left\| \sum_{n=1}^{N}T_{n}\left( U_{t}g^{p/2}h_{n}^{1/2}h_{n}^{1/2}g^{g/2}U_{t}^*\right) \right\|_{2}^{\frac{1}{2}} \left\|\sum_{n=1}^{N}T_{n}\left(U_{t}^*a^{p/2}h_{n}^{1/2}h_{n}^{\frac{1}{2}g^{p/2}}U_{t}\right) \right\|_{2}^{\frac{1}{2}} \nonumber\\
&\leq 4 \left\|U_{t}\left(g^{p/2}\sum_{n=1}^N~h_{n}~g^{p/2}\right)U_{t}^* \right\|_{2}^{\frac{1}{2}} \left\| U_{t}^*\left(g^{p/2}\sum_{n=1}^N~h_{n}~g^{p/2}\right)U_{t} \right\|_{2}^{\frac{1}{2}} \nonumber\\
&\leq 4\|g^p\|_{2}=4.
\end{align}

Then, by the three line lemma, we have
$$
\left\| \sum_{n=1}^{N} T_{n}g_{n} \right\|_{p}= \left\| \sum_{n=1}^N~T_{n}F(\theta) \right\|_{p} \leq 4^{\theta}.
$$
We complete the proof by applying the homogeneity property.
\end{proof}

Finally, we return to the proof of Theorem \ref{thm:main} by duality. 
\begin{proof}
  For $f\in L_p({\N})$, we have that $|f|\in  L_p({\N})$ and $|f|$ has   the same norm with $f$ by definition.  
  We apply \eqref{eq:dualityNorm2}  to the positive sequence in  $L_p({\N})$$T_n(|f|)$, and obtain
\begin{align*}
\|(T_{n}(|f|))\|_{L_{p}(\mathcal{N};\ell_{\infty})}&=\sup\left\{ \nu\left( \sum T_{n}(|f|)  y_{n} \right): y_{n}\geq 0,\|(y_{n})\|_{L_{q}(\mathcal{N};\ell_{1}^{N})}\leq 1\right\} \\
&= \sup\left\{ \nu\left( \sum |f| T_{n}(y_{n}) \right): y_{n}\geq 0,\|(y_{n})\|_{L_{q}(\mathcal{N};\ell_{1}^{N})}\leq 1\right\} \\
&\leq  \sup\left\{ \nu\left( |f|\sum   T_{n}y_{n} \right): z_{n} \geq 0,\|(z_{n})\|_{L_{q}(\mathcal{N};\ell_{1}^{N})}\leq 4^{2-\frac{2}{q}}\right\} \\
& \leq 4^{ \frac{2}{p}}\|f\|_{L_{p}(\mathcal{N})},
\end{align*} 
for all $2\leq p<\infty$.  By \eqref{eqinfty}, $\|(T_{n}(|f|))\|_{L_{p}(\mathcal{N};\ell_{\infty})}\leq 4^{1+\frac2p}\|f\|_{L_{p}(\mathcal{N})}$. By definition \eqref{supxn}, This means  that there exists $F\in L_p({\N}) $ such that $ \|F\|_{L_p({\N})}\leq 4^{1+\frac2p}\|f\|_{L_{p}(\mathcal{N})}$ and $T_n|f|\leq F$.  Theorem \ref{thm:main} follows since $$\frac1{2t}\int_{x-t}^{x+t}|f(y)|dy\leq 2T_n|f|$$ for every $2^{n-1}\leq t<2^{n}, n\in\Z$.
\end{proof}

\bibliographystyle{abbrv}

\end{document}